\documentclass{amsart}
\makeatletter
\def\LaTeX{\leavevmode L\raise.42ex
   \hbox{\kern-.3em\size{\sf@size}{0pt}\selectfont A}\kern-.15em\TeX}
\makeatother

\newcommand{\BibTeX}{{\rm B\kern-.05em{\sc
i\kern-.025emb}\kern-.08em\TeX}}

\newtheorem{theorem}{Theorem}[section]
\newtheorem{lemma}[theorem]{Lemma}

\theoremstyle{definition}
\newtheorem{definition}{Definition}

\newtheorem{remark}{Remark}
\numberwithin{equation}{section}

\begin{document}

\title[ Parseval frames and Weyl's law on sub-Riemannian manifolds]{Parseval space-frequency localized frames  on  sub-Riemannian compact homogeneous manifolds}

\author{Isaac Pesenson}

\maketitle

\begin{abstract}The objective of this chapter is to describe a  construction of Parseval bandlimited and localized frames on sub-Riemannian compact homogeneous manifolds.

\end{abstract}

\section{Introduction}\label{intro}
		
		The objective of this chapter is to describe a  construction of Parseval bandlimited and localized frames in $L_{2}$-spaces on a class of sub-Riemannian compact homogeneous manifolds. 
		
		The chapter begins with a brief review in section 2 of some results obtained in \cite{GP} where  a construction of Parseval bandlimited and localized  frames was performed  in $L_{2}({\bf M}),\>\>\>{\bf M}$ being a compact homogeneous manifold equipped with a natural Riemannian metric.

		In section \ref{2-sphere} we are using a sub-Riemannian structure on the  two-dimensional standard unit sphere $\mathbf{S}^{2}$ to explain the main differences between Riemannian and sub-Riemannian settings.  Each of these structures is associated with a distinguished second-order differential operator which arises from a metric.  These operators are  self-adjoint with respect to the usual normalized invariant (with respect to rotations) measure on $\mathbf{S}^{2}$.
		  The major difference between these operators is that  in the case of Riemannian metric the  operator is elliptic (the Laplace-Beltrami operator $\mathbf{L}$) and in the sub-Riemannian case it is not (the sub-Laplacian $\mathcal{L}$).  As a result, the  corresponding Sobolev spaces which are introduced as domains of powers of these operators are quite different. In the elliptic case one obtains the regular Sobolev spaces and in sub-elliptic one obtains  function spaces (sub-elliptic Sobolev spaces) in which functions have variable smoothness (compared to regular (elliptic) Sobolev smoothness).

	In section \ref{sub-Riem} we describe a class of sub-Riemannian structures on  compact homogeneous manifolds and consider a construction of  Parseval bandlimited and localized  frames associated with such structures. Leaving a detailed description of sub-Riemannian structures for later sections we will formulate our main result now.

		We consider compact homogeneous manifolds ${\bf M}$ equipped with  the so-called sub-Riemannian metric $\mu(x,y),\>x,y\in {\bf M}$ (see Definition \ref{metric}). 
To formulate our main result  we need a definition of a sub-Riemannian lattice on a manifold ${\bf M}$. The precise definitions of all the notions  used below will be given in the text.

\begin{lemma}\label{cover-0}
Let $ {\bf M}$ be a compact sub-Riemannian manifold and $\mu(x, y), \>x,y\in {\bf M}$  be a sub-Riemannian metric. Let  $B^{\mu}(x, r)$ be a ball in this metric with center $x\in {\bf M}$ and radius $r$.  There exists a natural number $N^{\mu}_{{\bf M}}$ such that for any sufficiently small $r>0$ 
there exists a set of points 
$\mathcal{M}_{r}^{\mu}=\{x_{i}\}$ with the following properties
\begin{enumerate}
\item  \textit {the balls $B^{\mu}(x_{i},  r/4)$ are disjoint,}
\item  \textit {the balls $B^{\mu}(x_{i}, r/2)$ form a cover of ${\bf M}$,}
\item  \textit {every point of ${\bf M}$ is covered by not more than  $N^{\mu}_{{\bf M}}$ balls $B^{\mu}(x_{i},r)$.}
\end{enumerate}
\end{lemma}

\begin{definition}
A set $\mathcal{M}_{r}^{\mu}=\{x_{i}\}$ constructed in the previous lemma will be called a metric $r$-lattice.
\end{definition}

The meaning of this definition is that points $\{x_{i}\}$ are distributed over ${\bf M}$  "almost uniformly" in the sense of the metric $\mu$.

We will consider compact homogeneous manifolds ${\bf M}={\bf G}/{\bf H}$ where ${\bf G} $ is a compact Lie group and ${\bf H}\subset {\bf G}$ is a closed subgroup.  Let $dx$ be an invariant (with respect to natural action of ${\bf G}$ on ${\bf M}$) measure on ${\bf M}$ and $L_{2}({\bf M})=L_{2}({\bf M}, dx)$ the corresponding Hilbert space of complex-valued functions on ${\bf M}$ with the inner product
$$
\left<f,g\right>=\int _{{\bf M}}f\overline{g}dx.
$$
The notation $\left|B^{\mu}(x, r)\right|$ will be used for the volume of the ball with respect to the measure $dx$. An interesting feature of sub-Riemann structures is that balls of the same radius may have essentially different volumes (in contrast to the case of the Riemann metric and Riemann measure). 

In the next Theorem 
we will mention  a sub-elliptic operator (sub-Laplacian) $\mathcal{L}$ (see the precise definition in (\ref{sub-L}))
which is hypoelliptic \cite{Hor},  self-adjoint and non-negative in $L_{2}(\bf {M})$. This operator is a natural analog of a Laplace-Beltrami operator in the case of a Riemannian  manifold.

	\begin{theorem}\label{main-1} We assume that ${\bf M}$ is a compact homogeneous manifold equipped with a sub-Riemann metric $\mu$ (see section \ref{sub-Riem}). Set $r_{j}=2^{-j-1},\>\>j=0, 1, 2, ...,$ and let $\mathcal{M}_{r_{j}}^{\mu}=\{x^{j}_{k}\}_{k=1}^{m_{j}}, \>\>\>x^{j}_{k} \in {\bf M},\>\>\>j=0,1,2,..$ be a sequence of metric lattices.

With every point $x_{k}^{j}$ one can associate a function $\Theta_{k}^{j}$ such that: 

\begin{enumerate}

\item  every $ \Theta^{j}_{k}$ is bandlimited in the sense that $ \Theta^{j}_{k}$ belongs to the space $\mathbf{E}_{[2^{2j-2}, 2^{2j+2}]}(\mathcal{L})$ which is the span of all eigenfunctions of $\mathcal{L}$ whose corresponding eigenvalues belong to the interval $[2^{2j-2}, 2^{2j+2})$,

\item every $ \Theta^{j}_{k}$ is essentially supported around $x_{k}^{j}$ in the sense that for any $N>0$ there exists a constant $C(N)>0$ such that for all $j,k$ one has
\begin{equation}\label{localiz}
 \left| \Theta^{j}_{k}(y)\right|\leq
    C( N)\left|B^{\mu}\left(x^{j}_{k}, 2^{-j}\right)\right|^{-1/2}\left(1+2^{j}\mu(x^{j}_{k},y)\right)^{-N},
\end{equation}

\item $\>\>\{\Theta^{j}_{k}\}\>\>$  is a Parseval frame i.e. for all $f\in L_{2}({\bf M})$
\begin{equation}\label{ tight}
 \sum_{j\geq 0}\>\> \sum_{1\leq k\leq m_{j}} \left|\left<f, \Theta^{j}_{k}\right>\right|^{2}= \|f\|^{2}_{L_{2}({\bf M})},
\end{equation}
 and as a  consequence of the Parseval property one has the following reconstruction formula
	\item  
	$$
	f= \sum_{j\geq 0}\>\> \sum_{1\leq k\leq m_{j}} \left<f, \Theta^{j}_{k}\right>\Theta^{j}_{k}.
$$
\end{enumerate}
\end{theorem}
	In Theorem \ref{beshom}  this frame is used to obtain characterization of sub-elliptic Besov spaces in terms of the frame coefficients.

\section{Parseval localized frames on Riemannian compact homogeneous manifolds}\label{ParsHM}

\subsection{Hilbert frames} Frames in Hilbert spaces were introduced in \cite{DS}.
\begin{definition}
A set of vectors $\{\psi_{v}\}$  in a Hilbert space $\mathcal{H}$ is called a frame if there exist constants $A, B>0$ such that for all $f\in \mathcal{H}$ 
\begin{equation}
A\|f\|^{2}_{2}\leq \sum_{v}\left|\left<f,\psi_{v}\right>\right|^{2}     \leq B\|f\|_{2}^{2}.
\end{equation}
The largest $A$ and smallest $B$ are called lower and upper frame bounds.  
\end{definition}

The set of scalars $\{\left<f,\psi_{v}\right>\}$ represents a set of measurements of a signal $f$. To synthesize the signal $f$ from this set of measurements one has to find another (dual) frame $\{\Psi_{v}\}$ and then a reconstruction formula is 
\begin{equation}\label{Hfr}
f=\sum_{v}\left<f,\psi_{v}\right>\Psi_{v}.
\end{equation}
Dual frames  are  not unique in general.  Moreover it is difficult   to find a dual frame.
However, for frames with $A=B=1$ the decomposition and synthesis of functions can be done with the same frame. In other words
 \begin{equation}
f=\sum_{v}\left<f,\psi_{v}\right>\psi_{v}.
\end{equation}
Such frames are known as Parseval frames.
For example,
 three vectors in $\mathbf{R}^{2}$ with angles  $2\pi/3$ between them whose lengths are all  $\sqrt{2/3}$ form a Parseval frame.

\subsection{ Compact homogeneous manifolds}

The basic information about compact homogeneous manifolds can be found in \cite{H1}, \cite{H2}.
A homogeneous compact manifold $\mathbf{{\bf M}}$ is a
$C^{\infty}$-compact manifold on which a compact
Lie group $\mathbf{G}$ acts transitively. In this case $\mathbf{{\bf M}}$ is necessarily of the form $\mathbf{G}/\mathbf{H}$,
where $\mathbf{H}$ is a closed subgroup of $\mathbf{G}$. The notation $L_{2}(\mathbf{{\bf M}}),$ is used for the usual Hilbert spaces, where $dx$ is the normalized  invariant measure on ${\bf M}$.

The best known example of such manifold is a unit sphere $\mathbf{S}^{n}$ in $\mathbf{R}^{n+1}$:  $\mathbf{S}^{n}=\mathbf{SO}(n+1)/\mathbf{SO}(n)=\mathbf{G}/\mathbf{H}.$

If $\mathbf{g}$ is the Lie algebra of a compact Lie group $\mathbf{G}$ then there exists a such choice of  basis
$X_{1},...,X_{d}$ in
$\mathbf{g}$,  for which  the operator
\begin{equation}\label{fulsumsquares}
-\mathbf{L}=X_{1}^{2}+X_{2}^{2}+\    ... +X_{d}^{2},    \ d=dim\ \mathbf{G}
\end{equation}
is a bi-invariant operator on $\mathbf{G}$.  Here $X_{j}^{2}$ is $X_{j}\circ X_{j}$ where we identify each $X_{j}$ with a left-invariant vector field on ${\bf G}$. We will use the same notation for its image under differential of the quasi-regular representation of $\mathbf{G}$ in $L_{2}({\bf M})$. This operator
$\mathbf{L}$, which is known as  the Casimir  operator is elliptic.
There are 
situations in which the operator $\mathbf{L}$ is, or is proportional to, the
Laplace-Beltrami operator  of an invariant metric on ${\bf M}$. This
happens for example, if ${\bf M}$ is a $n$-dimensional torus, a compact semi-simple
Lie group, or a compact symmetric space of rank one.

 Since ${\bf M}$ is compact and the operator $\mathbf{L}$   is elliptic it has a discrete spectrum $0=\lambda_{0}<\lambda_{1}\leq \lambda_{2}\leq......$ which goes to infinity without any accumulation points  and there exists a complete  family  $\{u_{j}\}$  of orthonormal eigenfunctions which form a  basis in $L_{2}({\bf M})$.

 The elliptic differential self-adjoint (in $L_{2}(\mathbf{M})$) operator $\mathbf{L}$ and its powers $\mathbf{L}^{s/2}, \  k\in {\bf R_{+}},$ can be extended from $C^{\infty}(\mathbf{M})$ to distributions.  The family  of Sobolev spaces $W_{p}^{s}({\bf M}),\ 1\leq p< \infty,\  s\in {\bf R},$ can be introduced as subspaces of $L_{p}({\bf M})$ with the norm 
\begin{equation}\label{BesselNorm}
\|f\|_{p}+\|\mathbf{L}^{s/2}f\|_{p}.
\end{equation}

One can show that when $s=k$ is a natural number    this norm
 is equivalent to the norm
\begin{equation}
|||f|||_{k,p}=\|f\|_{p}+\sum_{1\leq i_{1},..., i_{k}\leq
d}\|X_{i_{1}}...X_{i_{k}}f\|_{p},\ 1\leq p<\infty.
\end{equation}

We assume now that ${\bf M}$ is equipped with a ${\bf G}$-invariant Riemann metric $\rho$. The Sobolev spaces can also be  introduced in terms of local charts \cite{T}. 
We fix a  finite cover $\left \{B^{\rho}(y_{\nu}, r_{0})\right\}$ of ${\bf M}$ 
\begin{equation}
{\bf M}=\bigcup_{\nu} B^{\rho}(y_{\nu}, r_{0}),\label{cover}
\end{equation}
where $B^{\rho}(y_{\nu}, r_{0})$ is a  ball centered at $y_{\nu}\in {\bf M}$ of radius
$r_{0}$ contained in a coordinate chart.  Let consider $\Psi=\{\psi_{\nu}\}$ be a partition of unity
$\Psi=\{\psi_{\nu}\}$ subordinate to this cover. The Sobolev
spaces $W^{k}_{p}({\bf M}), k\in \mathbf{N}, 1\leq p<\infty,$ are
introduced as the completion of $C^{\infty}({\bf M})$ with respect
to the norm
\begin{equation}
\|f\|_{W^{k}_{p}({\bf M})}=\left(\sum_{\nu}\|\psi_{\nu}f\|^{p}
_{W^{k}_{p}(B^{\rho}(y_{\nu}, r_{0}))}\right) ^{1/p}.\label{Sobnorm}
\end{equation}

\begin{remark}
Spaces $W_{p}^{k}({\bf M})$ are independent of the choice of elliptic self-ajoint second order differential operator. For every choice of such operators corresponding norms (\ref{BesselNorm}) will be equivalent. Also, any two  norms of the form (\ref{Sobnorm}) are equivalent \cite{T}.  

\end{remark}

The Besov spaces can be introduced via  the formula
\begin{equation}\label{ellipticBesov}
\mathcal{B}_{p, q}^{\alpha }({\bf M})
:= \left (L_{p}({\bf M}),W^{r}_{p}({\bf M})\right)^{K}_{\alpha/r,q},
\end{equation}
where $
0<\alpha<r\in {\bf N},\ 1\leq p< \infty,\>\>\>1\leq q\leq \infty.
$
Here $K$ is the Peetre's interpolation functor.

An explicit norm in these spaces was given in \cite{Pes79}-\cite{Pes90b}.
For the same operators as above $\{X_{1},...,X_{d}\},\ d=dim \ {\bf G}$,  let $T_{1},..., T_{d}$
be the corresponding one-parameter groups of translation along integral
curves of the corresponding vector  fields i.e.
 \begin{equation}
 T_{j}(\tau)f(x)=f(\exp \tau X_{j}\cdot x),
 x\in {\bf M}, \tau \in \mathbf{R}, f\in L_{2}({\bf M});
 \end{equation}
 here $\exp \tau X_{j}\cdot x$ is the integral curve of the vector field
 $X_{j}$ which passes through the point $x\in {\bf M}$.
 The modulus of continuity is introduced as
\begin{equation}
\Omega_{p}^{r}( s, f)= $$ $$\sum_{1\leq j_{1},...,j_{r}\leq
d}\sup_{0\leq\tau_{j_{1}}\leq s}...\sup_{0\leq\tau_{j_{r}}\leq
s}\|
\left(T_{j_{1}}(\tau_{j_{1}})-I\right)...\left(T_{j_{r}}(\tau_{j_{r}})-I\right)f\|_{L_{p}({\bf M})},\label{M}
\end{equation}
where $f\in L_{p}({\bf M}),\ r\in \mathbf{N},  $ and $I$ is the
identity operator in $L_{p}({\bf M}).$   We consider the space of all functions in $L_{p}({\bf M})$ for which the
following norm is finite:
\begin{equation}
\|f\|_{L_{p}({\bf M})}+\left(\int_{0}^{\infty}(s^{-\alpha}\Omega_{p}^{r}(s,
f))^{q} \frac{ds}{s}\right)^{1/q} , 1\leq p<\infty, 1\leq q\leq \infty,\label{BnormX}
\end{equation}
with the usual modifications for $q=\infty$.

 \begin{theorem}
 The norm of the Besov space $B_p^{\alpha q}({\bf M})=(L_{p}({\bf M}),W^{r}_{p}({\bf M}))^{K}_{\alpha/r,q},\ 
 0<\alpha<r\in \mathbf{N},\ 
1\leq p<\infty, 1\leq q\leq \infty,$ is equivalent to the norm (\ref{BnormX}). Moreover, the norm
(\ref{BnormX}) is equivalent to the norm
\begin{equation}
\|f\|_{W_{p}^{[\alpha]}({\bf M})}+\sum_{1\leq j_{1},...,j_{[\alpha] }\leq d}
\left(\int_{0}^{\infty}\left(s^{[\alpha]-\alpha}\Omega_{p}^{1}
(s,X_{j_{1}}...X_{j_{[\alpha]}}f)\right)^{q}\frac{ds}{s}\right)^{1/q}\label{nonint}
\end{equation}
if $\alpha$ is not integer ($[\alpha]$ is its integer part).  If
$\alpha=k\in \mathbf{N}$ is an integer then the norm
(\ref{BnormX}) is equivalent to the norm (Zygmund condition)
\begin{equation}
\|f\|_{W_{p}^{k-1}({\bf M})}+ \sum_{1\leq j_{1}, ... ,j_{k-1}\leq d }
\left(\int_{0}^{\infty}\left(s^{-1}\Omega_{p}^{2}(s,
X_{j_{1}}...X_{j_{k-1}}f)\right)
 ^{q}\frac{ds}{s}\right)^{1/q}.\label{integer}
\end{equation}
\end{theorem}

\begin{definition}
The space of $\omega$-bandlimited functions $\mathbf{E}_{\omega}(\mathbf{L})$ is defined as the span of all eigenfunctions of $\mathbf{L}$ whose eigenvalues are not greater than $\omega.$
\end{definition}

To describe our  construction of frames  we need the notion of a lattice on a manifold ${\bf M}$ equipped with a Riemann metric $\rho$. This notion is similar to the corresponding notion introduced in Lemma \ref{cover-0}.

\begin{lemma}\label{cover}
If ${\bf M}$ is a compact Riemannian manifold  then there exists a natural $N^{\rho}_{{\bf M}}$ such that for any
sufficiently small $r$  there exists a set of points 
$\mathcal{M}^{\rho}_{r}=\{x_{i}\}$ with the following properties
\begin{enumerate}
\item  \textit {the balls $B^{\rho}(x_{i},  r/4)$ are disjoint,}
\item  \textit {the balls $B^{\rho}(x_{i}, r/2)$ form a cover of ${\bf M}$,}
\item  \textit {the height of the cover by the balls $B^{\rho}(x_{i},r)$ is not
greater than $N^{\rho}_{{\bf M}}.$}
\end{enumerate}
\end{lemma}

The meaning of this definition is that points $\{x_{k}\}$ distributed over ${\bf M}$  almost uniformly.

In \cite{GP}  the following theorem was proved for compact homogeneous manifolds considered with invariant Riemann metric.
		
	\begin{theorem}\label{main 10} Set $r_{j}=2^{-j-1},\>\>j=0, 1, 2, ...,$ and let $\mathcal{M}_{r_{j}}^{\rho}=\{x^{j}_{k}\}_{k=1}^{m_{j}}, \>\>\>x^{j}_{k} \in {\bf M},\>\>\>j=0,1,2,..$ be a sequence of metric lattices.

With every point $x_{k}^{j}$ we associate a function $\Psi_{k}^{j}$ such that: 

\begin{enumerate}

\item  every $\Psi_{k}^{j}$ is bandlimited in the sense that $ \Psi^{j}_{k}$ belongs to the space $\mathbf{E}_{[2^{2j-2}, 2^{2j+2}]}(\mathbf{L})$ which is the span of all eigenfunction of $\mathbf{L}$ whose corresponding eigenvalues belong to the interval $[2^{2j-2}, 2^{2j+2})$,

\item every $\Psi_{k}^{j}$ is essentially supported around $x_{k}^{j}$ in the sense that the following estimate holds for every $N >n$:
 \begin{equation}\label{sp-loc}
 \left| \Psi^{j}_{k}(y)\right|\leq
    C( N)
    2^{jn}\left(1+ 2^{j}\rho\ (x_{k}^{j}, y)\right)^{-N},\>\>\> dim \ {\bf M} = n,
    \end{equation}

\item $\>\>\{\Psi^{j}_{k}\}\>\>$  is a Parseval frame i.e. for all $f\in L_{2}({\bf M})$
\begin{equation}\label{ tight}
 \sum_{j\geq 0}\>\> \sum_{1\leq k\leq m_{j}} \left|\left<f, \Psi^{j}_{k}\right>\right|^{2}= \|f\|^{2}_{L_{2}({\bf M})},
 \end{equation}
 and
 \begin{equation}
 f= \sum_{j\geq 0}\>\> \sum_{1\leq k\leq m_{j}} \left<f, \Psi^{j}_{k}\right>\Psi^{j}_{k}.
 \end{equation}

\end{enumerate}

\end{theorem}

 As an important application of Theorem \ref{main 10}  one can describe Besov spaces in therms of the frame coefficients \cite{GP}.

\begin{theorem}\label{BesovR}
The norm of the Besov space 
$
\|f\|_{\mathcal{B}^{\alpha}_{p,q}({\bf M})} ,\>\>\>1\leq p< \infty, 0<q\leq\infty
$
is equivalent to the norm
$$
\|\tau(f)\|_{{\bf b}_{p,q}^{\alpha }}=
\left(\sum_{j = 0}^{\infty}
2^{jq(\alpha-n/p+n/2)}
\left(\sum_k |\langle f, \Psi^{j}_{k} \rangle|^p\right)^{q/p}\right)^{1/q}.
$$

\end{theorem}

\subsection{Example of $\mathbf{S}^{2}$ with Riemannian metric}

We consider ${\bf M}=\mathbf{S}^{2}$. In this case the Casimir operator  coincides with the Laplace-Beltrami operator $\mathbf{L}$ on $\mathbf{S}^{2}$ and it can be written as a sum of the vector fields on $\mathbf{S}^{2}$:
$$
\mathbf{L}=\sum_{i,j=1; i<j}^{3} X_{i,j}^{2}=\sum_{i,j=1; i<j}^{3}(x_{i}\partial_{x_{j}}-x_{j}\partial_{x_{i}})^{2}=\mathbf{L}.
$$

Let $\mathcal{P}_{l}$ denote the space of spherical harmonics of degree $l$, which are restrictions to ${\bf S}^{2}$ of harmonic homogeneous polynomials of degree $l$ in ${\bf R}^{3}$.

Each $\mathcal{P}_{l}$ is the eigenspace of  $\mathbf{L}$ that corresponds to the eigenvalue $-l(l+1)$. Let $\mathcal{Y}_{n,l},\>\>n=1,...,2l+1$ 
 be an orthonormal basis in $\mathcal{P}_{l}$. One has
$$
\mathbf{L}\mathcal{Y}_{m,l}=-l(l+1)\mathcal{Y}_{m,l}.
$$

Sobolev spaces $W_{p}^{k}(\mathbf{L}), 1\leq p< \infty $, can be introduced as usual by using a system of local coordinates or by using vector fields $X_{i,j}$:
\begin{equation}
\|f\|_{W_{p}^{k}({\bf M})}=\|f\|_{p}+\sum\sum\|X_{i,j}....X_{i,j}f\|_{p}
\end{equation}
Corresponding  Besov spaces $\mathcal{B}_{p,q}^{\alpha}(\mathbf{L})$ can be described either using  local coordinates or in terms of the modules of continuity constructed in terms of one-parameter groups of rotations $e^{\tau X_{i,j}}$  \cite{Pes79}-\cite{Pes90a}.
In particular,  when $p=2$ the Parseval identity for orthonormal bases and the theory of interpolation spaces imply  descriptions of the norms 
 of $W_{2}^{k}(\mathbf{L})$ and $\mathcal{B}_{2,2}^{\alpha}(\mathbf{L})$  in terms of Fourier coefficients:

\begin{equation}\label{norm1}
\left(\sum_{l=0}^{\infty}\sum_{n=1}^{2l+1}(l+1)^{2\alpha}|c_{n,l}(f)|^{2}\right)^{1/2},
\end{equation}
where
$$
c_{n,l}(f)=\int_{{\bf S}^{d}}f\mathcal{Y}_{n,l},\>\>\>f\in L_{2}({\bf S}^{d}).
$$

\section{Sphere $S^{2}$ with a sub-Riemannian metric. A sub-Laplacian and sub-elliptic spaces on $S^{2}$}\label{2-sphere}

To illustrate nature of sub-elliptic spaces we will consider the case of two-dimensional sphere ${\bf S}^{2}$.  We consider on ${\bf S}^{2}$ two vector fields $Y_{1}=X_{2,3}$ and $Y_{2}=X_{1,3}$ and the corresponding sub-Laplace operator 
$$
\mathcal{L}=Y_{1}^{2}+Y_{2}^{2}.
$$
Note that since  the operators $Y_{1},\> Y_{2}$  do not span the tangent space to ${\bf S}^{2}$ along a great circle with $x_{3}=0$ the operator $\mathcal{L}$ is not elliptic on $\mathbf{S}^{2}$.
However, this operator is hypoelliptic \cite{Hor} since $Y_{1},\>Y_{2}, $ and their commutator $Y_{3}=Y_{1}Y_{2}-Y_{2}Y_{1}=X_{1, 2}$ span the tangent space at every point of ${\bf S}^{2}$.

Let's compute its corresponding eigenvalues. 
In the standard spherical coordinates $(\varphi, \vartheta)$ spherical  harmonics $
\mathcal{Y}_{m,l}(\varphi, \vartheta),\>\>l=0,1,...,\>\>|m|\leq l$ are proportional to $e^{im\varphi}P_{l}^{m}(\cos\>\vartheta)$, where $P_{l}^{m}$ are associated Legendre polynomials. 
This representation shows that for $Y_{3}=X_{1,2}$ one has
$$
Y_{3}^{2}\mathcal{Y}_{m,l}=-m^{2}\mathcal{Y}_{m,l}.
$$
Since $\mathcal{Y}_{m,l}$ is an eigenfunction of $\mathbf{L}$ with the eigenvalue $-l(l+1)$ we obtain
$$
\mathbf{L}\mathcal{Y}_{m,l}=-l(l+1)\mathcal{Y}_{m,l}
$$
and
$$
\mathcal{L}\mathcal{Y}_{m,l}=\mathbf{L}\mathcal{Y}_{m,l}-Y_{3}^{2}\mathcal{Y}_{m,l}=-\left(l(l+1)-m^{2}\right)\mathcal{Y}_{m,l}.
$$
It shows that spherical functions are eigenfunctions of both $\mathcal{L}$ and $\mathbf{L}$.

The graph norm of  a fractional power of $\mathcal{L}$  is equivalent to the norm 
\begin{equation}\label{norm2}
\left(\sum_{l=0}^{\infty}\sum_{|m|\leq l}\left((l+1)^{2}-m^{2}\right)^{\alpha}|c_{m,l}(f)|^{2}\right)^{1/2},
$$
$$
c_{m,l}(f)=\int_{{\bf S}^{d}}f\mathcal{Y}_{m,l},\>\>\>f\in L_{2}(\mathbf{L}).
\end{equation}
Note that these spaces $W_{2}^{\alpha}(\mathcal{L})$ are exactly the Besov spaces $\mathcal{B}^{\alpha}_{2, 2}(\mathcal{L})$.

We introduce subelliptic (anisotropic) Sobolev space $W_{2}^{\alpha}(\mathcal{L}),\>\>\alpha\geq 0,$ as the domain of $\mathcal{L}^{\alpha}$ with the graph norm and define Besov spaces $\mathcal{B}_{2,q}^{\alpha}(\mathcal{L})$ as

$$
\mathcal{B}^{\alpha}_{2,q}(\mathcal{L})=( L_{2}({\bf S}^{2}),W^{r}_{2}(\mathcal{L}))^{K}_{\theta, q},\>\>\> 0<\theta=\alpha/r<1,\>\>\>
1\leq q\leq \infty.
$$
where $K$ is the Peetre's interpolation functor.

Note that vector fields $Y_{1}, Y_{2}$  span  the tangent space to $\mathbf{S}^{2}$ at every point away from a great circle  $x_{3}=0$. For this reason around such points a function belongs to the domain of $
\mathcal{L}$ if and only if it belongs to the regular Sobolev space $W_{2}(\mathbf{L})$.

At the same time  the fields $Y_{1},\> Y_{2}$   do not span the tangent space to ${\bf S}^{2}$ along a great circle with $x_{3}=0$. However, the fields $Y_{1}, Y_{2}$ and their commutator $  Y_{3}=Y_{1}Y_{2}-Y_{2}Y_{1}=X_{1, 2}$ do span the tangent space along $x_{3}=0$.  This fact  implies that along   the circle $x_{3}=0$,              functions in the spaces $W_{2}^{r}(\mathcal{L})$ and  $
\mathcal{B}^{\alpha}_{2,q}(\mathcal{L})$ are loosing $1/2$ in smoothness compared to their smoothness at other points on ${\bf S}^{2}$. 
In other words, the following embeddings hold true
$$
W_{2}^{\alpha}(\mathbf{L})\subset W_{2}^{\alpha}(\mathcal{L})\subset W_{2}^{\alpha/2}(\mathbf{L}),
$$
$$
\mathcal{B}^{\alpha}_{2,q}({\mathbf{L}})\subset \mathcal{B}^{\alpha}_{2,q}(\mathcal{L})\subset \mathcal{B}^{\alpha/2}_{2,q}({\mathbf{L}}),
$$
which follow from a much more general results in \cite{RS}, \cite{NSW}, \cite{Pes90a}.

We would like to stress that subelliptic function spaces are different from the usual (elliptic) spaces.  For example, if  $W_{2}^{\alpha}(\mathbf{L}) $ is the regular Sobolev space than  general theory implies the embeddings 
$$
W_{2}^{\alpha}(\mathcal{L})\subset W_{2}^{\alpha/2}(\mathbf{L}),\>\>\>\mathcal{B}^{\alpha}_{2,q}(\mathcal{L})\subset \mathcal{B}^{\alpha/2}_{2,q}(\mathbf{L}).
$$
As the following Lemma shows, these embeddings are generally sharp.
\begin{lemma}
For every $\alpha>0$ and $\delta>\alpha/2$ there exists a function that belongs to $
W_{2}^{\alpha}(\mathcal{L})$ but does not belong to $W_{2}^{\delta}(\mathbf{L}).
$

\end{lemma}
\begin{proof}
For a $\delta>\alpha/2>0$ pick any $ \gamma$ that satisfies the inequalities
$$
-\frac{1}{2}-\delta<\gamma<-\frac{1}{2}-\frac{\alpha}{2}
$$

Let $c_{n,l}$ be a sequence such that $c_{n,l}=0$ if $n\neq l$ and $c_{l,l}=(2l+1)^{\gamma}$. For a function with such Fourier coefficients the norm (\ref{norm2}) is finite  since
$$
\sum_{l=0}^{\infty}(2l+1)^{\alpha}(2l+1)^{2\gamma}=\sum_{l=0}^{\infty}(2l+1)^{\alpha+2\gamma}<\infty,\>\>\> \alpha+2\gamma<-1,
$$
but the norm (\ref{norm1}) is infinite
$$
\sum_{l=0}^{\infty}(2l+1)^{2\delta}(2l+1)^{2\gamma}=\sum_{l=0}^{\infty}(2l+1)^{2(\delta+\gamma)},\>\>\>2(\delta+\gamma)>-1.
$$

\end{proof}

\section{Sub-Riemannian structure on compact homogeneous manifolds}\label{sub-Riem}

Let     $\mathbf{M}=\mathbf{G}/\mathbf{H}$ be a compact homogeneous manifold and      ${\bf X}=\{X_{1},\ ...,X_{d}\}$ be a basis of the Lie algebra $\mathbf{g}$, the same as in (\ref{fulsumsquares}). Let 
\begin{equation}\label{hvf}
{\bf Y}=\{Y_{1},...,Y_{m}\}
\end{equation}
 be a subset of ${\bf X}=\{X_{1},\ ...,X_{d}\}$ such that $Y_{1},...,Y_{m}$ and all their commutators 
\begin{equation}\label{com}
Y_{j,k}=[Y_{j}, \>Y_{k}]=Y_{j}Y_{k}-Y_{k}Y_{j},\>\>\>
$$
$$
Y_{j_{1},...,j_{n}}=[Y_{j_{1}},
[....[Y_{j_{n-1}}, Y_{j_{n}}]...]],
\end{equation}
 of order $n\leq Q$ span the entire algebra $\mathbf{g}$.
 Let 
\begin{equation}\label{vf}
Z_{1}=Y_{1},  Z_{2}=Y_{2},  ... , Z_{m}=Y_{m}, \>\>\>... \>\>\>, Z_{N},
\end{equation}
be an enumeration of all commutators (\ref{com}) up to order $n\leq Q$. If a $Z_{j}$ corresponds to a commutator of length $n$ we say that $deg(Z_{j})=n$.

Images of vector fields (\ref{vf}) under the natural projection $p: \mathbf{G}\rightarrow \mathbf{M}=\mathbf{G}/\mathbf{H}$ span the tangent space to $\mathbf{M}$ at every point and will be denoted by the same letters.

\begin{definition}
A sub-Riemann structure on $\mathbf{M}=\mathbf{G}/\mathbf{H}$ is defined as a set of vectors fields on $\mathbf{M}$ which are images of the vector fields (\ref{hvf}) under the projection $p$. They can also be identified with differential operators in $L_{p}({\bf M}),\>1\leq p<\infty,$ under the quasi-regular representation of ${\bf G}$.

\end{definition}

One can define  a non-isotropic metric $\mu$ on ${\bf M}$ associated with the fields $\{Y_{1},...,Y_{m}\}$.

\begin{definition}\cite{NSW}\label{metric}
Let $C(\epsilon)$ denote the class of absolutely continuous mappings $\varphi: [0,1]\rightarrow {\bf M}$ which almost everywhere satisfy the differential equation 
$$
\varphi^{'}(t)=\sum_{j=1}^{m}b_{j}(t)Z_{j}(\varphi(t)),
$$
where $|b_{j}(t)|<\epsilon^{deg(Z_{j})}$. Then we define 
$
\mu(x,y)$ as the lower bound of all such $\epsilon>0$ for which there exists $\varphi \in C(\epsilon)$ with $\varphi(0)=x,\> \varphi(1)=y$.
\end{definition}

The corresponding family of balls in ${\bf M}$ is given by 
$$
B^{\mu}(x,\epsilon)=\{y\in {\bf M} : \ \mu(x,y)<\epsilon\}.
$$
These balls reflect the non-isotropic nature of the vector fields $Y_{1},...,Y_{m}$ and their commutators.  For a small $\epsilon>0$  ball $B^{\mu}(x,\epsilon)$ is of size $\epsilon$ in the directions $Y_{1},...,Y_{m}$, but only of size $\epsilon^{n}$  in the directions of commutators of length $n$.

 It is known \cite{NSW}  that the following  property holds for certain $c=c(Y_{1},...,Y_{m}),\> C=C(Y_{1},...,Y_{m})$:
$$
c\rho(x,y)\leq \mu(x,y)\leq C\left(\rho(x,y)\right)^{1/Q}
$$
where $\rho$ stands for an $\mathbf{G}$-invariant Riemannian metric on ${\bf M}=\mathbf{G}/\mathbf{H}$.
We will be interested in  the following sub-elliptic operator (sub-Laplacian)
\begin{equation}\label{sub-L}
-\mathcal{L}=Y_{1}^{2}+...+Y_{m}^{2}
\end{equation}
which is hypoelliptic \cite{Hor} self-adjoint and non-negative in $L_{2}(\bf {M})$.

\begin{definition}
The space of $\omega$-bandlimited functions $\mathbf{E}_{\omega}(\mathcal{L})$ is defined as the span of all eigenfunctions of $\mathcal{L}$ whose eigenvalues are not greater than $\omega.$
\end{definition}

Due to the uncertainty principle bandlimited functions in $\mathbf{E}_{\omega}(\mathcal{L})$ are not localized on ${\bf M}$ in the sense that their supports coincide with ${\bf M}$. 

Using the operator $\mathcal{L}$ we define non-isotropic Sobolev spaces $W_{p}^{k}(\mathcal{L}),\>\>1\leq p<\infty, $ and non-isotropic Besov spaces $\mathcal{B}^{\alpha}_{p,q}(\mathcal{L}),\> \>1\leq p<\infty,\>1\leq q\leq \infty,$    by using formulas (\ref{BesselNorm}) and (\ref{ellipticBesov}) respectively.

\section{Product property for subelliptic Laplace operator}\label{product}

The results of this section   play a crucial role in our construction of the Parseval frames. 
In what follows we consider previously defined operators 
$$
-\mathbf{L}=X_{1}^{2}+X_{2}^{2}+\    ... +X_{d}^{2},    \ d=dim\ \mathbf{G},
$$
and
$$
-\mathcal{L}=Y_{1}^{2}+...+Y_{m}^{2},    \ m<d,
$$
as differential operators in $L_{2}(\bf {M})$.
\begin{lemma}\cite{GP}\label{prod}
\label{prodlem}
If ${\bf M}={\bf G}/{\bf H}$ is a compact homogeneous manifold then for any $f$ and $g$ in 
 ${\mathbf E}_{\omega}(\mathbf{L})$,  their product $fg$ belongs to
${\mathbf E}_{4d\omega}(\mathbf{L})$, where $d$ is the dimension of the
group ${\bf G}$.

\end{lemma}

\begin{proof}

For every $X_{j}$ one has
$$
X_{j}^{2}(fg)=f(X_{j}^{2}g)+2(X_{j}f)(X_{j}g)+g(X_{j}^{2}f).
$$

Thus, the function ${\mathbf{L}}^{k}\left(fg\right)$ is a sum of
$(4d)^{k}$ terms of the form
$$
(X_{i_{1}}...X_{i_{m}}f)(X_{j_{1}}...X_{j_{2k-m}}g).
$$
This implies that
\begin{equation}
\left\|\mathbf{L}^{k}\left(fg\right)\right\|_{\infty}\leq
(4d)^{k}\sup_{0\leq m\leq 2k}\sup_{x,y\in
{\bf M}}\left|X_{i_{1}}...X_{i_{m}}f(x)\right|\left|X_{j_{1}}...X_{j_{2k-m}}g(y)\right|.\label{estim3}
\end{equation}
Let us show that for all $f,g \in {\bf E}_{\omega}({\mathbf{L}})$ the following inequalities hold:

\begin{equation} \label{estim1}
\|X_{i_{1}}...X_{i_{m}}f\|_{L_{2}({\bf M})}\leq
\omega^{m/2}\|f\|_{L_{2}({\bf M})}
\end{equation}

and
\begin{equation}
\|X_{j_{1}}...X_{j_{2k-m}}g\|_{L_{2}({\bf M})}\leq
\omega^{(2k-m)/2}\|g\|_{L_{2}({\bf M})}\label{estim2}.
\end{equation}
By construction (see (\ref{fulsumsquares})) the operator
$
-{\mathbf{L}}=X_{1}^{2}+\ ...+X_{d}^{2}
$
commutes with every $X_{j}$ and the same is
true for $(-\mathbf{L})^{1/2}$. From here  one can  obtain the
following equality:
\begin{equation}
\|{\mathbf{L}}^{s/2}f\|_{L_{2}({\bf M})}^{2}=\sum_{1\leq i_{1},...,i_{s}\leq
d}\|X_{i_{1}}...X_{i_{s}}f\|_{L_{2}({\bf M})}^{2},\ s\in \mathbf{N}, \label{eq0}
\end{equation}
which implies the estimates (\ref{estim1}) and (\ref{estim2}). 
The formula
(\ref{estim3}) along with the formula 
\begin{equation}
\|{\mathbf{L}}^{m/2}f\|_{L_{2}({\bf M})}\leq
\omega^{m/2}\|f\|_{L_{2}({\bf M})}.\label{estim4}
\end{equation}
 imply the
estimate
\begin{equation}
\|{\mathbf{L}}^{k}(fg)\|_{L_{2}({\bf M})}\leq (4d)^{k}\sup_{0\leq m\leq
2k}\|X_{i_{1}}...X_{i_{m}}f\|_{L_{2}({\bf M})}\|X_{j_{1}}...X_{j_{2k-m}}g\|_{\infty}\leq
$$
$$(4d)^{k}\omega^{m/2}\|f\|_{L_{2}({\bf M})}\sup_{0\leq m\leq
2k}\|X_{j_{1}}...X_{j_{2k-m}}g\|_{\infty}.
\end{equation}
Using the Sobolev embedding Theorem and the elliptic regularity of
$\mathbf{L}$, we obtain for every $s>\frac{dim {\bf M}}{2}$
\begin{equation}
\|X_{j_{1}}...X_{j_{2k-m}}g\|_{\infty}\leq
C({\bf M})\|X_{j_{1}}...X_{j_{2k-m}}g\|_{W_{2}^{s}({\bf M})}\leq
$$
$$
C({\bf M})\left\{\|X_{j_{1}}...X_{j_{2k-m}}g\|_{L_{2}({\bf M})}+
\|\mathbf{L}^{s/2}X_{j_{1}}...X_{j_{2k-m}}g\|_{L_{2}({\bf M})}\right\},
\end{equation}
where $W_{2}^{s}({\bf M})$ is the Sobolev space of $s$-regular functions on
${\bf M}$. The estimate (\ref{estim4}) gives the following
inequality:
\begin{equation}
\|X_{j_{1}}...X_{j_{2k-m}}g\|_{\infty}\leq
C({\bf M})\left\{\omega^{k-m/2}\|g\|_{L_{2}({\bf M})}+\omega^{k-m/2+s}\|g\|_{L_{2}({\bf M})}\right\}\leq
$$
$$
C({\bf M})\omega^{k-m/2}\left\{\|g\|_{L_{2}({\bf M})}+\omega^{s/2}\|g\|_{L_{2}({\bf M})}\right\}=
C({\bf M},g,\omega,s)\omega^{k-m/2},\>\>\>s>\frac{dim\  {\bf M}}{2}.
\end{equation}
Finally we have the following estimate:
\begin{equation}
\|\mathbf{L}^{k}(fg)\|_{L_{2}({\bf M})}\leq
C({\bf M},f,g,\omega,s)(4d\omega)^{k},\>\>\>s>\frac{dim \ {\bf M}}{2},\>\>k\in
\mathbf{N},
\end{equation}
which leads to our result.

\end{proof}

\begin{lemma}\label{EF}
There exist positive $c,\>C$ such that for $\omega>1$ the following embeddings hold
\begin{equation}
{\bf E}_{\omega}(\mathcal{L})\subset {\bf E}_{c\omega^{Q}}(\mathbf{L}) ,
\end{equation} 

\begin{equation}
 {\bf E}_{\omega}(\mathbf{L}) \subset {\bf E}_{C\omega}(\mathcal{L}).
\end{equation} 
\end{lemma}
\begin{proof}
There exists a constant $a=a(\mathbf{L},\>\mathcal{L})$ such that for all $f$ in the Sobolev space $W_{2}^{Q}(\mathbf{M})$ \cite{NSW}
$$
\|\mathbf{L}f\|\leq a\|(I+\mathcal{L})^{Q}f\|.
$$
Since $\mathbf{L}$ belongs to the center of the enveloping algebra of the Lie algebra $\mathbf{g}$ it commutes with $\mathcal{L}$. Thus one has for sufficiently smooth $f$:
$$
\|\mathbf{L}^{l}f\|\leq a^{l}\|(I+\mathcal{L})^{Q l}f\|,\>\> l\in \mathbf{R}.
$$
It implies that  if $f\in {\bf E}_{\omega}(\mathcal{L})$, then for $\omega\geq 1$
$$
\|\mathbf{L}^{l}f\|\leq a^{l}\|(I+\mathcal{L})^{Q l}f\|\leq \left(a(1+\omega)^{Q}\right)^{l}\|f\|\leq \left(2a\omega^{Q}\right)^{l}\|f\|,\>\> l\in \mathbf{R},
$$
which shows that $f\in {\bf E}_{2a\omega^{Q}}(\mathbf{L})$.
Conversely, since for some $b=b(\mathbf{L},\>\mathcal{L})$
$$
\|\mathcal{L}f\|\leq b\|(I+\mathbf{L})f\|,\>\>f\in W_{2}^{2}(\mathbf{M}),
$$
we have
$$
\|\mathcal{L}^{l}f\|\leq b^{l}\|(I+\mathbf{L})^{l}f\|,\>\>f\in W_{2}^{2l}(\mathbf{M}),
$$
and for $f\in {\bf E}_{\omega}(\mathbf{L})$
$$
\|\mathcal{L}^{l}f\|\leq b^{l}\|(I+\mathbf{L})^{l}f\|\leq \left(b(1+\omega)\right)^{l}\|f\|\leq (2b\omega)^{l}\|f\|,\>\>f\in W_{2}^{2l}(\mathbf{M}).
$$
\end{proof}

The product property of bandlimited functions is described in the following Theorem.
\begin{theorem}
\label{product} 
There exists a constant $C_{0}=C_{0}(\mathcal{L})>0$ such that for any $f,\>g\in {\bf E}_{\omega}(\mathcal{L})$ the product $fg$ belongs to ${\bf E}_{C_{0}\omega^{Q}}(\mathcal{L})$.
\end{theorem}

\begin{proof}
If $f,g\in {\bf E}_{\omega}(\mathcal{L})$ then $f,g\in {\bf E}_{c\omega^{Q}}(\mathbf{L})$. According to Lemma \ref{prod}  their product $fg$ belongs to ${\bf E}_{4dc\omega^{Q}}(\mathbf{L})$ which implies that for some $C_{0}=C_{0}(\mathcal{L})$ the product $fg$ belongs to ${\bf E}_{C_{0}\omega^{Q}}(\mathcal{L})$.

\end{proof}

\section{Positive cubature formulas on sub-Riemannian manifolds}

Now we are going to prove existence of cubature formulas which are exact on $\mathbf{E}_{\omega}(\mathcal{L})$,
and have positive coefficients of the right size.

Let $\mathcal{M}_{r}=\{x_{k}\}$  be a $r$-lattice and $\{B^{\mu}(x_{k},r)\}$ be an associated family of balls that satisfy only properties (1) and (2) of  Lemma \ref{cover-0}. 
We define
$$
U_{1}=B^{\mu}(x_{1}, r/2)\setminus \cup_{i,\>i\neq 1}B^{\mu}(x_{i}, r/4),
$$
and
\begin{equation}\label{U}
U_{k}=B^{\mu}(x_{k},  r/2)\setminus \left(\cup_{j<k}U_{j}\cup_{i,\>i\neq k}B^{\mu}(x_{i},  r/4)\right).
\end{equation}
One can verify the following properties.
\begin{lemma} The sets $\left\{U_{k}\right\}$ form a disjoint measurable cover (up to a set of measure zero) of $\mathbf{M}$ and
\begin{equation}\label{disjcover}
B^{\mu}(x_{k}, r/4)\subset U_{k}\subset B^{\mu}(x_{k}, r/2)
\end{equation}
\end{lemma}

 We have the following Plancherel-Polya inequalities  \cite{Pes00}, \cite{Pes04b}.

\begin{theorem}\label{PPT}
There exist  positive constants  $a_{1}=a_{1}(\mathbf{M, Y}), a_{2}=a_{2}(\mathbf{M, Y})$, and $a_{0}=a_{0}(\mathbf{M,Y})$ such that, if for a given  $\omega>0$ one has
\begin{equation}\label{rate}
0<r<a_{0}\omega,
\end{equation}
then for any metric $r$-lattice $\mathcal{M}_{r}=\{x_{k}\}$  the following inequalities hold
\begin{equation}
a_{1}\sum_{k}|U_{k}||f(x_{k})|^{2}\leq\|f\|_{L_{2}(\mathbf{M})}
\leq a_{2}\sum_{k} |U_{k}||f(x_{k})|^{2}, \label{PP}
\end{equation}
for every $f\in \mathbf{E}_{\omega}(\mathcal{L}).$
\end{theorem}

\begin{proof}
One has 

$$
|f(x)|\leq |f(x_{k})|+|f(x)-f(x_{k})|,
$$
$$
\int_{U_{k}}|f(x)|^{2}dx\leq 2\left(|U_{k}||f(x_{k})|^{2}+\int_{U_{k}}|f(x)-f(x_{k})|^{2}dx\right),
$$
and
$$
\|f\|^{2}\leq \sum_{k}\int_{U_{k}}|f(x)|^{2}dx\leq 2\left( \sum_{k}|U_{k}||f(x_{k})|^{2}+ \sum_{k}\int_{U_{k}}|f(x)-f(x_{k})|^{2}dx\right).
$$

Take an $X\in   \mathbf{g},\>\>|X|=1,     $ for which $ \exp\ tX\cdot x_{k} =x$ for some $t\in \mathbf{R}$. Since every such vector field (as a field on  $\mathbf{M}$)  is a linear combination   of the fields $[Y_{i_{1}},...[Y_{i_{l-1}}, Y_{i_{l}}]...], 1\leq l\leq Q, 1\leq i_{j}\leq m$, the Newton-Leibniz formula applied to a smooth $f$ along the corresponding integral curve joining  $x$ and $x_{k}$ gives
$$
|f(x)-f(x_{k})|^{2}\leq Cr^{2}\sum_{l=1}^{Q} \sum _{1\leq i_{1},i_{2},... i_{l}\leq m}
 \left(\sup_{y\in B^{\mu}(x_{k},r/2)}\left|Y_{i_{1}}Y_{i_{2}}... Y_{i_{l}}f(y)\right|\right)^{2}.
 $$
Applying anisotropic version of the Sobolev inequality \cite{NSW} we
obtain

$$
|f(x)-f(x_{k})|^{2}\leq C r^{2}\sum_{l=1}^{Q} \sum_{1\leq i_{1},i_{2},... i_{l}\leq m}\left(\sup_{y\in
B^{\mu}(x_{k},r/2)}\left|Y_{i_{1}}Y_{i_{2}}... Y_{i_{l}}f(y)\right|\right)^{2}\leq
$$
 $$
 C r^{2} \sum^{Q}_{l=0}\sum _{1\leq i_{1},i_{2}, ...
,i_{l}\leq m}\|
Y_{i_{1}}Y_{i_{2}}...Y_{i_{l}}f\|^{2}_{H^{Q/2+\varepsilon}(B^{\mu}(x_{k}, r/2))},$$
where $ x\in U_{k},\>\>\varepsilon > 0,\>\>
C=C(\varepsilon).$ Next, 
$$
\sum_{k}\int _{B^{\mu}(x_{k}, r/2)}|f(x)-f(x_{k})|^{2}dx \leq 
$$
 $$
 C r^{n+2} \sum^{Q}_{l=0}\sum _{1\leq i_{1},i_{2}, .. i_{l}\leq m}\sum_{k}\|Y_{i_{1}}... Y_{i_{l}}f\|^{2}_{H^{Q/2+\varepsilon}(B^{\mu}(x_{k}, r/2))}\leq
 $$ 
 $$ 
 C r^{n+2}\sum^{Q}_{l=0}\sum _{1\leq i_{1},...,i_{l}\leq m}\|Y_{i_{1}}... Y_{i_{l}}f\|^{2}_{H^{Q/2+\varepsilon}(\mathbf{M})}\leq 
  C r^{n+2}        \left(\|f\|^{2}+\|\mathcal{L}^{Q}f\|^{2}\right).
 $$
 All together we obtain the inequality
 $$
 \|f\|^{2}\leq 2 \sum_{k}|U_{k}||f(x_{k})|^{2}+Cr^{n+2} \left(\|f\|^{2}+\|\mathcal{L}^{Q}f\|^{2}\right).
 $$
Note that for $f\in \mathbf{E}_{\omega}(\mathcal{L})$
$$
\|\mathcal{L}^{Q}f\|\leq C\omega^{Q}\|f\|.
$$
 Thus, if for a given $\omega>0$ we pick an $r>0$  a way that 
 $$
 Cr^{n+2}(1+\omega)^{Q}<1
 $$
 then for a certain $C_{1}=C_{1}(M)>0$ one obtains the right-hand side of (\ref{PP})
 $$
  \|f\|^{2}\leq C_{1} \sum_{k}|U_{k}||f(x_{k})|^{2}.
 $$
 The left-hand side of (\ref{PP}) follows from the Sobolev and Bernstein inequalities. 
\end{proof}

The Plancherel-Polya inequalities (\ref{PP}) can be used to prove the so-called sub-elliptic positive cubature formula. The proof goes along the same lines as in \cite{GP}, \cite{pg},  (see also \cite{FM}, \cite{CKP}).

The  precise   statement is the following.

\begin{theorem} 
\label{cubature}
There exists a constant $a=a(\mathbf{M, Y})>0$ such that for a given $\omega>0$  if $r=a\omega^{-1}$ then for any $r$-lattice $ \mathcal{M}_{r}=\{x_{k}\}$ 
there exist strictly positive coefficients $\{\alpha_{k}\}$, \  for which the following equality holds for all functions in $ \mathbf{E}_{\omega}(\mathcal{L})$:
\begin{equation}
\label{cubway}
\int_{\mathbf{M}}fdx=
\sum_{k}f(x_{k}) \alpha_{k} .
\end{equation}
Moreover, there exists constants  $\  b_{1}>0, \  b_{2}>0, $  such that  the following inequalities hold:
\begin{equation}\label{setsmeasures}
b_{1}|U_{k}|\leq  \alpha_{k}\leq b_{2}|U_{k}|,
\end{equation}
where the sets $U_{k}$ are defined in (\ref{U}).
\end{theorem}

\section{Space localization of kernels}

According to the spectral theorem if $F$ is a Schwartz function on the line, then there is a well defined operator $F(\mathcal{L})$ in the space $L_{2}(\mathbf{\mathbf{M}})$ such that for any $f\in L_{2}(\mathbf{\mathbf{M}})$ one has 
\begin{equation}\label{function}
\left(F(\mathcal{L})f\right)(x)=\int_{\mathbf{\mathbf{M}}}\mathcal{K}^{F}(x,y)f(y)dy,
\end{equation}
where $dy$ is the invariant normalized measure on $\mathbf{\mathbf{M}}$. If $\left\{\lambda_{j}\right\} $ and $\left\{u_{j}\right\}$ are sets of eigenvalues and eigenfunctions of $\mathcal{L}$ respectively then 
\begin{equation}\label{kernel-0}
\mathcal{K}^{F}(x,y)=\sum_{j=0}^{\infty}F(\lambda_{j})u_{j}(x)\overline{u_{j}}(y).
\end{equation}
We will be especially interested in operators of the form $F(t^{2}\mathcal{L})$, where $F$ is a Schwartz function and $t>0$. The corresponding kernel will be denoted as $\mathcal{K}_{t}^{F}(x,y)$ and 
\begin{equation}\label{kernel-1}
\mathcal{K}_{t}^{F}(x,y)=\sum_{j=0}^{\infty}F(t^{2}\lambda_{j})u_{j}(x)\overline{u_{j}}(y).
\end{equation}
Note, that variable  $t$ here is a kind of scaling parameter.

The following important estimate was proved in \cite{CKP} in the setting of the so-called Dirichlet spaces. It is a consequence of the main result in \cite{M} that sub-Riemannin manifolds we consider in our article  are the Dirichlet spaces.
\begin{theorem}
If $F\in C_{0}^{\infty}(\mathbf{R})$ is even than for every $N>2Q$ there exists a $C_{N}=C_{N}(F,\mathbf{M, Y})>0$ such that 
\begin{equation}\label{local}
\left| \mathcal{K}_{t}^{F}(x,y)\right|\leq C_{N}\left(\left|B^{\mu}(x, t)\right|\left|B^{\mu}(y, t)\right|\right)^{-1/2}\left(1+t^{-1}\mu(x,y)\right)^{-N},\>\>\>0<t\leq 1.
\end{equation}

\end{theorem}

\section{ Parseval space-frequency localized frames on sub-Riemannian manifolds and proof of Theorem \ref{main-1}}

 Let $g\in C^{\infty}(\mathbf{R}_{+})$ be a monotonic function with support in $ [0,\>  2^{2}], $ and $g(s)=1$ for $s\in [0,\>1], \>0\leq g(s)\leq 1, \>s>0.$ Setting  $G(s)=g(s)-g(2^{2}s)$ implies that $0\leq G(s)\leq 1, \>\>s\in supp\>G\subset [2^{-2},\>2^{2}].$  Clearly, $supp\>G(2^{-2j}s)\subset [2^{2j-2}, 2^{2j+2}],\>j\geq 1.$ For the functions
 $
 F_{0}(s)=\sqrt{g(s)}, \>\>F_{j}(s)=\sqrt{G(2^{-2j}s)},\>\>j\geq 1, \>\>\>
 $
 one has $\sum_{j\geq 0}F_{j}^{2}(s)=1, \>\>s\geq 0$.
 Using the spectral theorem for $\mathcal{L}$ one  can define bounded self-adjoint operators $F_{j}(\mathcal{L})$ as
 $$
 F_{j}(\mathcal{L})f(x)=\int_{\mathbf{\mathbf{M}}}\mathcal{K}^{F}_{2^{-j}}(x,y)f(y)dy,
 $$
where 
\begin{equation}\label{kernel}
\mathcal{K}^{F}_{2^{-j}}(x,y)= \sum_{\lambda_{m}\in [2^{2j-2}, 2^{2j+2}]} F(2^{-2j}\lambda_m) u_m(x) \overline{u_m(y)}.
\end{equation}
The same spectral theorem implies  
$
\sum_{j\geq 0} F_{j}^2(\mathcal{L})f = f,\>\>f \in  L_{2}(\mathbf{\mathbf{M}}),
$
and taking inner product with $f$ gives
\begin{equation}
\label{norm equality-0}
\|f\|^2=\sum_{j\geq 0}\left< F_{j}^2(\mathcal{L})f,f\right>=\sum_{j\geq 0}\|F_{j}(\mathcal{L})f\|^2 .
\end{equation}
 Moreover, since the function $  F_{j}(s)$ has its support in  $
[2^{2j-2},\>\>2^{2j+2}]$ the functions $ F_{j}(\mathcal{L})f $ are bandlimited to  $
[2^{2j-2},\>\>2^{2j+2}]$.

Next, consider the sequence $\omega_{j}=2^{2j+2},\>j=0, 1, ....\>$.  
By (\ref{norm equality-0}) the equality  $
\|f\|^2=\sum_{j\geq 0}\|F_{j}(\mathcal{L})f\|^2 $ holds, where every  function $ F_{j}(\mathcal{L})f $ is bandlimited to  $
[2^{2j-2},\>\>2^{2j+2}]$.
Since for every $
\overline{F_{j}( \mathcal{L})f} \in \mathbf{\mathbf{E}}_{2^{2j+2}}(\mathcal{L})$
one can use  Theorem \ref{product}  to conclude that
$$
|F_{j}( \mathcal{L})f|^2\in  \mathbf{\mathbf{E}}_{C_{0}2^{Q(2j+2)}}(\mathcal{L}).
$$
According to Theorem \ref{cubature}  there exists a constant $a=a(\mathbf{M, Y})>0$ such that for all natural  $j$ if
\begin{equation}
\label{rate}
r_j = b2^{-Q(j+1)},\>\>b=aC_{0},
\end{equation}
then for any  $r_{j}$-lattice $\mathcal{M}_{r_{j}}$ one can find positive coefficients $\alpha_{j,k}$ with 
for which the following exact cubature formula holds
\begin{equation}
\label{samrate}
\|F_{j}(\mathcal{L})f\|^2_2 = \sum_{k=1}^{K_j}\alpha_{j,k}\left|F_{j}(\mathcal{L})f(x_{j,k})\right|^2,
\end{equation}
where $x_{j,k} \in \mathcal{M}_{r_{j}}$, $k = 1,\ldots,K_j = card\>(\mathcal{M}_{r_{j}})$.
Using the kernel $\mathcal{K}_{2^{-j}}^{F}$  of the operator $F_{j}(\mathcal{L})$
we  define the functions
\begin{equation}
\label{vphijkdf}
\Theta_{j,k}(y) =  \sqrt{\alpha_{j,k}}\>\overline{\mathcal{K}^{F}_{2^{-j}}}(x_{j,k},y) = 
$$
$$
\sqrt{\alpha_{j,k}} \sum_{\lambda_{m}\in [2^{2j-2}, 2^{2j+2}]} \overline{F}(2^{-2j}\lambda_m) \overline{u}_m(x_{j,k}) u_m(y).
\end{equation}
One can easily see that for every  $f \in L_2(\mathbf{\mathbf{M}})$ the equality 
$ \|f\|^2_2 = \sum_{j,k} |\langle f,\Theta_{j,k} \rangle|^2$ holds. Moreover, the first two items of Theorem \ref{main-1} are also satisfied.  Thus, Theorem \ref{main-1} is proven.

As an application one can obtain  description of sub-elliptic Besov spaces $\mathcal{B}_{p,q}^{\alpha }(\mathcal{L}),\>\>1\leq p<\infty,\>1\leq q\leq \infty,$ in terms of the Fourier coefficients with respect to this frame $\left\{\Theta_{j,k}\right\}$. 

Consider the quasi-Banach space ${\bf b}_{p,q}^{\alpha }$ which consists of
 sequences
$s=\{s^j_k\}$ ($j \geq 0,\ 1 \leq k \leq {\mathcal K}_j$)
 satisfying
\begin{equation}
\label{sjkbes}
\|s\|_{{\bf b}_{p,q}^{\alpha }}=\left(\sum_{j \geq 0}^{\infty} 2^{j\alpha q} \left(\sum_k\left| B^{\mu}(x_{k}^{j},\>2^{-j})\right|^{1/p-1/2} |s^j_k|^p\right)^{q/p}\right)^{1/q} < \infty,
\end{equation}
 and introduce  the following mappings
\begin{equation}\label{tau}
\tau(f) = \{\langle f, \Theta^{j}_{k}\rangle\},
\end{equation}
and
 \begin{equation}\label{sigma}
\sigma(\{s^j_k\}) = \sum_{j\geq 0}^{\infty}\sum_k  s^j_k  \Theta^{j}_{k}.
\end{equation}
It is not difficult to prove the following result (see \cite{GP} for the Riemann case).
\begin{theorem}
\label{beshom}
Let  $\Theta^{j}_{k}$ be the same as above. Then
 for $1\leq p< \infty,\>\>0<q\leq \infty,\>\>\alpha>0$
the following statements are valid:
\begin{enumerate}
\item  $\tau$ in (\ref{tau}) is a well defined bounded operator $\tau: \mathcal{B}_{p,q}^{\alpha }(\mathcal{L}) \to
{\bf b}_{p,q}^{\alpha}$;
\item $\sigma$ in (\ref{sigma}) is a well defined bounded operator $\sigma: {\bf b}_{p,q}^{\alpha } \to \mathcal{B}_{p,q}^{\alpha }(\mathcal{L})$;
\item  $\sigma \circ \tau = id$;
\end{enumerate}

Moreover,  the following norms are equivalent: 
$$
\|f\|_{\mathcal{B}^{\alpha}_{p,q}(\mathcal{L})}  \asymp \|\tau(f)\|_{{\bf b}_{p,q}^{\alpha }},
$$
where
$$
\|\tau(f)\|_{{\bf b}_{p,q}^{\alpha }}=
\left(\sum_{j \geq 0}^{\infty} 2^{j\alpha q} \left(\sum_k\left| B^{\mu}(x_{k}^{j},\>2^{-j})\right|^{1/p-1/2} |s^j_k|^p\right)^{q/p}\right)^{1/q}.
$$

The constants in these norm equivalence   relations
can be estimated uniformly over compact ranges of the parameters $p,q,\alpha$.

\end{theorem}

I am thankful to Hartmut F\"{u}hr and Gerard Kerkyacharian for stimulating discussions.

\end{document}